\documentclass{amsart}
\usepackage{amsmath}
\usepackage{amscd}
\usepackage{amssymb}
\usepackage{amsfonts}
\newtheorem{theorem}{Theorem}[section]
\newtheorem{lemma}[theorem]{Lemma}
\newtheorem{corollary}[theorem]{Corollary}

\newtheorem{proposition}[theorem]{Proposition}

\theoremstyle{definition}

\theoremstyle{remark}
\newtheorem{remark}[theorem]{Remark}
\numberwithin{equation}{section}
\newcommand{\abs}[1]{\lvert#1\rvert}

\begin{document}
\title[Some extensions of
the mean curvature flow in Riemannian manifolds] {Some extensions of
the mean curvature flow\\ in Riemannian manifolds}

\author{Jia-Yong Wu}

\address{Department of Mathematics, East China Normal
University, Dong Chuan Road 500, Shanghai 200241, P. R. China}

\address{\emph{Present Address}: Department of Mathematics, Shanghai Maritime University,
Shanghai 200135, P. R. China}

\email{jywu81@yahoo.com}

\thanks{This work is partially supported by the NSFC10871069.}


\dedicatory{}


\begin{abstract}
Given a family of smooth immersions $F_t: M^n\to N^{n+1}$ of closed
hypersurfaces in a locally symmetric Riemannian manifold $N^{n+1}$
with bounded geometry, moving by the mean curvature flow, we show
that at the first finite singular time of the mean curvature flow,
certain subcritical quantities concerning the second fundamental
form blow up. This result not only generalizes a recent result of
Le-Sesum and Xu-Ye-Zhao, but also extends the latest work of N. Le
in the Euclidean case (arXiv: math.DG/1002.4669v2).
\end{abstract}
\maketitle

\section{Introduction}\label{sec1}
We study the blow-up phenomena of the geometric quantities for
compact hypersurfaces $M^n$, $n\geq 2$, without boundary, which are
smoothly immersed in a Riemannian manifold $N^{n+1}$. Let $M^n=M_0$
be given by some diffeomorphism
\[
F_0: U\subset \mathbb{R}^n\rightarrow F_0(U)\subset M_0\subset
N^{n+1}.
\]
Consider a smooth one-parameter family of immersions
\[
 F(\cdot, t): M^{n}\rightarrow N^{n+1}
\]
satisfying the evolution equation
\begin{equation}\label{MCF1}
\frac{\partial}{\partial t}F(x,t)=-H(x,t)\nu(x,t),\quad \forall (x,
t)\in M\times [0,T)
\end{equation}
with the initial condition
\[
F(\cdot, 0)=F_{0}(\cdot).
\]
Here $H(x,t)$ and $\nu(x, t)$ denote the mean curvature and the unit
outward normal vector of the hypersurface $M_t=F(M^n,t)$ at the
point $F(x,t)$. Equation \eqref{MCF1} is often called the mean
curvature flow (MCF for short). We easily see that \eqref{MCF1} is a
quasilinear parabolic equation with a smooth solution at least on
some short time interval.

When the ambient space is the Euclidean space $\mathbb{R}^{n+1}$, G.
Huisken \cite{Huisken84} showed that the norm of the second
fundamental form $A(\cdot,t)$ of an evolving hypersurface under the
mean curvature flow must blow up at a finite singular time.

\begin{theorem}[G. Huisken \cite{Huisken84}]\label{HuiAound}
Suppose $T<\infty$ is the first singularity time for a compact mean
curvature flow \eqref{MCF1} with the ambient space
$N^{n+1}=\mathbb{R}^{n+1}$. Then $\sup_{M_t}|A|(\cdot,t)\to \infty$
as $t\to T$.
\end{theorem}

Later, G. Huisken and C. Sinestrari \cite{HS} studied the blow up of
$H$ near a singularity for mean convex hypersurfaces. They also
established lower bounds on the principal curvatures in this
mean-convex setting.

After that, N. Le and N. Sesum \cite{Le-Sesum}, and H.-W. Xu, F. Ye
and E.-T. Zhao \cite{XYZ1} used different methods to extend the
previous results. They independently showed that if the second
fundamental form is uniformly bounded from below and the mean
curvature is bounded in certain integral sense, then the mean
curvature flow can be extended smoothly past time $T$. Their proofs
are both based on blow-up arguments, and the Moser iteration using
the Michael-Simon inequality \cite{[Michael-Simon]}.

Meanwhile N. Le and N. Sesum \cite{Le-Sesum,Le-Sesum2}, and H.-W.
Xu, F. Ye and E.-T. Zhao \cite{XYZ1} gave another global conditions
for extending a smooth solution to the mean curvature flow.
Specifically, they proved that

\begin{theorem}\label{dglobscal}
Suppose $T<\infty$ is the first singularity time for a mean
curvature flow \eqref{MCF1} of compact hypersurfaces of
$N^{n+1}=\mathbb{R}^{n+1}$. Let $p$ and $q$ be positive numbers
satisfying $\frac np+\frac 2q=1$. If
\[
||A||_{L^{p, q}(M\times [0,T))}:
=\left[\int^T_0\left(\int_{M_t}|A|^p\right)^{q/p}dt\right]^{1/q}<
\infty,
\]
then this flow can be extended past time $T$. In particular, let
$p=q=n+2$, and if
\[
\int_0^{T}\int_{M_t}|A|^{n+2}d\mu dt<\infty,
\]
then this flow can be extended past time $T$.
\end{theorem}
Note that this result is still based on a blow-up argument, using
the compactness theorem for the  mean curvature flow ~\cite{ChenHe}.
Recently, N. Le \cite{Le} gave a logarithmic improvement of the
above results by proving that a family of subcritical quantities
concerning the second fundamental form  blows up at the first finite
singular time of the mean curvature flow.

\begin{theorem}[N. Le \cite{Le}]\label{gbsclog}
Suppose $T<\infty$ is the first singularity time for a mean
curvature flow \eqref{MCF1} of compact hypersurfaces of
$N^{n+1}=\mathbb{R}^{n+1}$. If
\[
\int_0^T\int_{M_t}\frac{|A|^{n+2}}{\log(1+|A|)}d\mu dt<\infty,
\]
then this flow can be extended past time $T$.
\end{theorem}

If the ambient space $N$ is a general Riemannian manifold, G.
Huisken \cite{Huisken86} gave an sufficient condition to assure the
extension over time for mean curvature flow \eqref{MCF1} provided
the ambient space $N$ admits bounded geometry. Here we say that a
Riemannian manifold is called bounded geometry if its sectional
curvature and the first covariant derivative of the curvature tensor
are bounded, and if the injective radius is bounded from below by a
positive constant. In \cite{Cooper}, A.A. Cooper mainly investigated
the behaviour at finite-time singularities of the mean curvature
flow of compact Riemannian submanifolds $M^n\to N^{n+\alpha}$
($\alpha\geq 1$) and generalized Huisken's results. He observed that
in fact it suffices to consider the tensor $\bar{A}_{ij}=H^\alpha
h_{ij\alpha}$, where $H=tr A$ is the mean curvature and $h$ are the
components of the second fundamental form. Namely, he proved that

\begin{theorem}[A.A. Cooper \cite{Cooper}]\label{Coopound}
Let $(N,h)$ be a Riemannian manifold with bounded geometry. Suppose
$T<\infty$ is the first singular time for a mean curvature flow
\eqref{MCF1} of compact submanifolds of $(N,h)$. Then
$\sup_{M_t}|\bar{A}|(\cdot,t) \to \infty$ as $t\to T$.
\end{theorem}
In other words, from Theorem~\ref{Coopound}, we can see that the
geometric quantity $\bar{A}(\cdot,t)$ uniformly bounded is enough to
extend a mean curvature flow of compact Riemannian submanifolds
$M^n\to N^{n+\alpha}$, provided $N$ is a bounded geometry. Cooper's
proof relies strongly on a blow-up argument combined with a
compactness property of the mean curvature flow.

In the case of $\alpha=1$, Cooper's result has been improved in
\cite{XYZ2}, assuming certain integral bounds on the mean curvature
and the second fundamental form is bounded from below.
\begin{theorem}[H.-W. Xu, F. Ye and E.-T. Zhao \cite{XYZ2}]\label{globsc2}
Let $F_t:M^n\rightarrow N^{n+1}$ $(n\geq 3)$ be a solution to the
compact mean curvature flow \eqref{MCF1} of on a finite time
interval $[0,T)$. If there is a positive constant $C$ such that
\[
h_{ij}\geq-C
\]
for $(x,t)\in M\times [0,T)$. Further we assume that
\[||H||_{L^\alpha(M\times [0,T))}:=\left(\int^T_0\int_{M_t}|H|^\alpha d\mu
dt\right)^{\frac{1}{\alpha}}<\infty
\] for some $\alpha\geq n+2$.
Then this flow can be extended past time $T$.
\end{theorem}

In this paper, we want to extend Theorems \ref{dglobscal} and
\ref{gbsclog} to the case when the ambient space is a general
Riemannian manifold. Similar to the logarithmic improvement result
proved by N. Le \cite{Le}, we obtain a family of subcritical
quantities involving the second fundamental form blows up at the
first finite singular time of the mean curvature flow. In
particular,
\begin{theorem}\label{wuN2}
Suppose $T<\infty$ is the first singularity time for a mean
curvature flow \eqref{MCF1} of compact hypersurfaces of a locally
symmetric Riemannian manifold $N^{n+1}$ with bounded geometry. If
\[
\int_0^T\int_{M_t}\frac{|A|^{n+2}}{\log(a+|A|)}d\mu dt<\infty,
\]
where $a$ is a fixed constant and $a\geq 1$, then this flow can be
extended past time $T$.
\end{theorem}
Obviously, Theorem \ref{wuN2} includes Theorem \ref{gbsclog} proved
by N. Le in \cite{Le}. To prove Theorem \ref{wuN2}, we mainly follow
the ideas of the proof of Theorem \ref{gbsclog} in \cite{Le}, where
a blow up argument and the Moser iteration are employed. In our
setting, one difference of the proof of Theorem \ref{gbsclog} is
that the Michael-Simon inequality used in the Euclidean case should
be replaced by the Hoffman-Spruck Sobolev inequality for
submanifolds of a Riemannian manifold (see Lemma \ref{Hoff-Spr}
below). Another difference is that curvature of the ambient space
$N^{n+1}$ will interfere with the evolution of the second
fundamental form on $M^n$, and will may further interfere the Moser
iteration process. However, if we assume the ambient space satisfies
some special restrictions, then we can still go through along the
Le's proof of Theorem \ref{gbsclog}.

Moreover, as a corollary Theorem \ref{wuN2} implies the following
result.

\begin{corollary}\label{wuN1}
Suppose $T<\infty$ is the first singularity time for a mean
curvature flow \eqref{MCF1} of compact hypersurfaces of a locally
symmetric Riemannian manifold $N^{n+1}$ with bounded geometry. If
\[
\int_0^T\int_{M_t}|A|^\alpha d\mu dt<\infty,
\]
for some $\alpha\geq n+2$, then this flow can be extended past time
$T$.
\end{corollary}
\begin{remark}
Corollary \ref{wuN1} generalizes a recent result of H.-W. Xu, F. Ye
and E.-T. Zhao (see Theorem 1.1 in \cite{XYZ1}). When the ambient
space $N$ is the Euclidean space, we recover their result. Moreover
Corollary \ref{wuN1} also extends a special case proved by N. Le and
N. Sesum \cite{Le-Sesum2} (see Theorem \ref{dglobscal} above).
\end{remark}

Besides the above works, the closest precedent for the mean
curvature flow is the Ricci flow introduced by R.S. Hamilton in his
seminal paper \cite{Hamilton} (see also \cite{[CLN]}), where
Hamilton showed that if $T<\infty$ is the maximal existence time of
a closed Ricci flow solution $g(t)$, $t\in [0,T)$, then the supremum
of the Riemannian curvature blows up as $t\rightarrow T$. In other
words, a uniform bound for the Riemannian curvature on a finite time
interval $M\times[0,T)$ is enough to extend Ricci flow past time
$T$. Later, N. Sesum \cite{Sesum} employed Perelman's no local
collapsing theorem \cite{[Perelman]} and improved the Hamilton's
extension result. That is, if the norm of Ricci curvature is
uniformly bounded over a finite time interval $[0,T)$, then we can
extend the flow smoothly past time $T$. In \cite{Wang}, B. Wang
improved the Sesum's result by a blow-up argument and the Moser
iteration. He showed that if Ricci curvature is uniformly bounded
from below and the scalar curvature is bounded in certain integral
sense on a finite time interval $[0,T)$, then the Ricci flow can be
extended past time $T$. Most recently, N. Le and N. Sesum in
\cite{Le-Sesum3}, and J. Enders, R. M\"{u}ller and P. Topping in
\cite{EMT} independently showed that the Type I Ricci flow can be
extended past time $T$ if the scalar curvature is uniformly bounded
over a finite time interval $[0,T)$.

The rest of this paper is organized as follows. In Sect.~\ref{sec2},
we will give some basic notation and preliminary results. In
Sect.~\ref{sec3}, we will obtain Sobolev inequalities for the mean
curvature flow in Riemannian manifolds.  In Sect.~\ref{sec4}, we
will give the reverse H\"{o}lder and Harnack inequalities for a
subsolution to a parabolic equation evolving by the mean curvature
flow in Riemannian manifolds. In Sect.~\ref{sec5}, we will prove
Proposition \ref{smallcont}, which means the second fundamental form
can be bounded by its certain integral sense. In Sect.~\ref{sec7},
we will apply Proposition \ref{smallcont} to finish the proof of
Theorem \ref{wuN2}. In Sect.~\ref{sec8}, we will prove Corollary
\ref{wuN1} using Theorem \ref{wuN2}.


\section{Preliminaries}\label{sec2}
In the following sections, Latin indices range from $1$ to $n$,
Greek indices range from $0$ to $n$ and the summation convention is
understood. Let
\[
F(\cdot, t): M^{n}\rightarrow N^{n+1}
\]
be a one-parameter family of smooth hypersurfaces immersed in a
Riemannian manifold $N$ satisfying the evolution equation
(\ref{MCF1}). We denote the induced metric and the second
fundamental form on $M$ by $g=\{g_{ij}\}$ and $A=\{h_{ij}\}$. The
mean curvature of $M$ is the trace of the second fundamental form,
$H=g^{ij}h_{ij}$. The square of the second fundamental form is
denoted by $|A|=g^{ij}g^{kl}h_{ik}h_{jl}=h_{ik}h^{ik}$. We write
$\bar{R}m=\{\bar{R}_{\alpha\beta\gamma\delta}\}$ and
$\bar{\nabla}\bar{R}m=\{\bar{\nabla}_\sigma\bar{R}_{\alpha\beta\gamma\delta}\}$
for the curvature tensor of $N$ and its covariant derivative. Let
$\nu$ be the outer unit normal to $M_t$. Then for a fixed time $t$,
we can choose a local field of frame $e_0, e_1, \cdots, e_n$ in $N$
such that when restricted to $M_{t}$, we have $e_0=\nu$, $e_i
=\frac{\partial F}{\partial x_i}$. The relations between
$A=\{h_{ij}\}$, $Rm$ and $\bar{R}m$ are then given by the following
equations of Gauss and Codazzi:
\begin{equation}\label{Gauss}
R_{ijkl} = \bar{R}_{ijkl}+h_{ik}h_{jl}-h_{il}h_{jk},
\end{equation}
\begin{equation}\label{eq-cod2}
\nabla_k h_{ij}-\nabla_j h_{ik}=\bar{R}_{0ijk}.
\end{equation}
In \cite{Huisken86}, we have the following evolution equations:
\begin{eqnarray*}
\frac{\partial}{\partial t}g_{ij}&=&-2Hh_{ij},\\
\frac{\partial}{\partial t}H&=&\triangle
H+H\left(|A|^2+\bar{R}ic(\nu,\nu)\right),\\
\frac{\partial}{\partial t}|A|^2&=&\triangle|A|^2-2|\nabla A|^2
+2|A|^2(|A|^2+\bar{R}ic(\nu,\nu))\\
&&-4(h^{ij}h^m_j\bar{R}_{mli}\
^l-h^{ij}h^{lm}\bar{R}_{milj})\\
&&-2h^{ij}(\bar{\nabla}_j\bar{R}_{0li}\
^l+\bar{\nabla}_l\bar{R}_{0ij}\ ^l),
\end{eqnarray*}
where $\bar{R}ic(\nu,\nu)=\bar{R}_{0l0}\ ^{l}$.

\section{Sobolev inequalities for the MCF}\label{sec3}

In this section we first introduce the Hoffman-Spruck Sobolev
inequality. Then we apply it to prove the following important
Lemma~\ref{lemm1}, which is essential in the proof of a Sobolev
inequality, Proposition~\ref{propo1zh}, for the mean curvature flow
\eqref{MCF1}. More importantly, this Sobolev inequality will be
crucial for the Moser iteration in the next sections.

Now we present the following Hoffman-Spruck Sobolev inequality for
Riemannian submanifolds in~\cite{[Hoffman-Spruck]}.

\begin{lemma}[D. Hoffman and J. Spruck~\cite{[Hoffman-Spruck]}]
\label{Hoff-Spr} Let $M\to N$ be an isometric immersion of
Riemannian manifolds of dimension $n$ and $n+p$, $(p\geq 1)$,
respectively. Some notations are adopted as before. Assume $K_N\leq
b^2$ and let $h$ be a non-negative $C^1$ function on $M$ vanishing
on $\partial M$. Then
\begin{equation}\label{sobolev}
\left(\int_{M} h^{n/(n-1)}dV_{M}\right)^{(n-1)/n}\leq c(n)\int_{M}
\left[|\nabla h|+h|H|\right]dV_{M},
\end{equation}
provided
\begin{equation}\label{condition1}
b^2(1-\alpha)^{-2/n}\left(\omega_n^{-1}Vol(\mathrm{supp}
h)\right)^{2/n}\leq 1
\end{equation}
and
\begin{equation}\label{condition2}
2\rho_0\leq\bar{R}(M),
\end{equation}
where
\begin{equation} \label{canshu1}
\rho_0=\left\{
\begin{aligned}
        b^{-1}\sin^{-1}\left[b(1-\alpha)^{-1/n}\left(\omega_n^{-1}Vol(\mathrm{supp}
h)\right)^{1/n}\right]\quad\quad&\mathrm{for}\,\,\, b\,\,\, \mathrm{real},\\
(1-\alpha)^{-1/n}\left(\omega_n^{-1}Vol(\mathrm{supp}h)\right)^{1/n}
\quad\quad&\mathrm{for}\,\,\, b\,\,\, \mathrm{imaginary}.
\end{aligned}
\right.
\end{equation}
Here $\alpha$ is a free parameter, $0<\alpha<1$, and
\begin{equation}\label{canshu2}
c(n):=c(n,\alpha)=\pi\cdot
2^{n-1}\alpha^{-1}(1-\alpha)^{-1/n}\frac{n}{n-1}\omega_n^{-1/n}.
\end{equation}
\end{lemma}

\begin{remark}
In Lemma \ref{Hoff-Spr}, we may replace the assumption $h\in C^1(M)$
by $h\in W^{1,1}(M)$. As the mentioned remark
in~\cite{[Hoffman-Spruck]}, the optimal choice of $\alpha$ to
minimize $c$ is $\alpha=n/(n+1)$. When $b$ is real we may replace
condition~\eqref{condition2} by the stronger condition
$\bar{R}\geq\pi b^{-1}$. When $b$ is a pure imaginary number and the
Riemannian manifold $N$ is simply connected and complete,
$\bar{R}(M)=+\infty$. Hence conditions~\eqref{condition1}
and~\eqref{condition2} are automatically satisfied.
\end{remark}

Following the proof of Lemma 2.1 in~\cite{Le} by means of Lemma
\ref{Hoff-Spr}, we have the following general result.
\begin{lemma}\label{lemm1}
Let $M$ be a compact $n$-dimensional hypersurface without boundary,
which is smoothly embedded in $N^{n+1}$. Assume $K_N\leq b^2$. Let
\begin{equation}
Q = \left\{ \begin{aligned}
\frac{n}{n-2} &\quad \text{if}~ n>2\\
<\infty &\quad \text{if} ~ n=2
\end{aligned}
\right.
\end{equation}
Then, for all non-negative Lipschitz functions $v$ on $M$, we have
\[
\|v\|^2_{L^{2Q}(M)}\leq c_n\left(\|\nabla v\|^2_{L^{2}(M)}+
\|H\|^{\frac{2(n+3)}{3}}_{L^{n+3}(M)}\|v\|^2_{L^2(M)}\right)
\]
provided the function $h:=v^{\frac{2(n-1)}{n-2}}$ satisfies
conditions~\eqref{condition1} and~\eqref{condition2}, where $H$ is
the mean curvature of $M$ and $c_n$ is a positive constant depending
only on $n$.
\end{lemma}

\begin{remark}
In Lemma~\ref{lemm1}, when $b$ is a pure imaginary number and the
Riemannian manifold $N$ is simply connected and complete,
$\bar{R}(M)=+\infty$. Hence conditions~\eqref{condition1}
and~\eqref{condition2} are automatically satisfied.
\end{remark}

Similar to the proof of Proposition 2.1 in \cite{Le}, using
Lemma~\ref{lemm1} and H\"{o}lder's inequality, our Sobolev type
inequality for the mean curvature flow in Riemannian manifolds is
stated in the following proposition.

\begin{proposition}\label{propo1zh}
For all non-negative Lipschitz functions $v$, we have
\begin{equation*}
\begin{aligned}
&||v||^{\beta}_{L^{\beta}(M\times[0,T))}\\
&\leq c_n\max_{0\leq t\leq T}||v||^{4/n}_{L^2(M_t)}\left(||\nabla
v||^2_{L^2(M\times[0,T))}+\max_{0\leq t\leq T}||v||^2_{L^2(M_t)}
||H||^{\frac{2(n+3)}{3}}_{L^{n+3}(M\times[0,T))}\right)
\end{aligned}
\end{equation*}
provided the function $h:=v^{\frac{2(n-1)}{n-2}}$ satisfies
conditions~\eqref{condition1} and~\eqref{condition2} for any
$t\in[0,T)$, where $\beta:=\frac{2(n+2)}{n}$.
\end{proposition}


\section{Reverse H\"{o}lder and Harnack inequalities}\label{sec4}
In this section we can follow the lines of~\cite{Le}
or~\cite{Le-Sesum}, and easily obtain a soft version of reverse
H\"{o}lder inequality and a Harnack inequality for parabolic
inequality during the mean curvature flow in Riemannian manifolds.
Consider the following differential inequality
\begin{equation}\label{keyeq}
\left(\frac{\partial}{\partial t}-\Delta\right) v\leq (f+C)v, ~v\geq
0,
\end{equation}
where the function $f$ has bounded $L^q(M\times [0, T))$-norm with
$q>\frac{n+2}{2}$, and $C$ is a fixed positive constant. Let
$\eta(t,x)$ be a smooth function with the property that $\eta(0,x) =
0$ for all $x$.
\begin{lemma}\label{sorevho}
Let
\begin{equation}\label{Czero}
C_0\equiv C_0(q)=||f||_{L^q(M\times[0,T))},\qquad
C_1=\left(1+||H||^{\frac{2(n+3)}{3}}_{L^{n+3}(M\times[0,
T))}\right)^{\frac{n}{n+2}},
\end{equation}
where $\beta>1$ be a fixed number and $q>\frac{n+2}{2}$. Then, there
exists a positive constant $C_a=C_a(n,q,C_0,C_1,C)$ such that
\begin{equation}
\begin{aligned}\label{xianzhi1}
 &||\eta^2 v^{\beta}||_{L^{\frac{n+2}{n}}(M\times[0,T))}\\
&\leq C_a\Lambda(\beta)^{1+\nu}\left|\left|v^{\beta}\left(\eta^2 +
|\nabla\eta|^2+2\eta \left|\left(\frac{\partial}{\partial
t}-\Delta\right)\eta\right|\right)\right|\right|_{L^1(M\times[0,T))},
\end{aligned}
\end{equation}
provided the function $(\eta
v^{\frac{\beta}{2}})^{\frac{2(n-1)}{n-2}}$ satisfies
conditions~\eqref{condition1} and~\eqref{condition2} for any
$t\in[0,T)$, where
\begin{equation}\label{nueq}
\nu=\frac{n+2}{2q-(n+2)},
\end{equation}
and $\Lambda(\beta)$ is a positive constant depending on $\beta$
such that $\Lambda(\beta)\geq 1$ if $\beta\geq 2$.
\end{lemma}
\begin{remark}
In fact, in Lemma \ref{sorevho} we can choose
\begin{equation}\label{Ca}
C_a(n,q,C_0,C_1,C)=(2c_nC_0C_1)^{1+\nu}+2c_nC_1(C+1).
\end{equation}
\end{remark}
\begin{proof}[Proof of Lemma~\ref{sorevho}]
We mainly follow the ideas of the proof of Lemma 3.1 in~\cite{Le}.
First, we use $\eta^2 v^{\beta-1}$ as a test function in the
following inequality
\[
-\Delta v+\frac{\partial v}{\partial t}\leq (f+C)v.
\]
Namely, for any $s\in (0, T]$,
\begin{equation}\label{int1}
\int_0^s\int_{M_t}(-\Delta v)\eta^2 v^{\beta-1}d\mu dt+
\int_0^s\int_{M_t} \frac{\partial v}{\partial t}\eta^2
v^{\beta-1}d\mu dt\leq\int_0^s\int_{M_t}|f+C|\eta^2 v^{\beta}d\mu
dt.
\end{equation}
Note that, using the integration by parts we have
\begin{equation}\label{int2}
\int_{M_t}(-\Delta v)\eta^2 v^{\beta-1}d\mu=\int_{M_t}2\langle\nabla
v,\nabla\eta\rangle\eta v^{\beta-1} d\mu +(\beta-1)\int_{M_t}\eta^2
v^{\beta-2}|\nabla v|^2 d\mu.
\end{equation}
By the properties of $\eta$ above, we get
\begin{equation}
\begin{aligned}\label{int3}
\int_0^s\int_{M_t}\frac{\partial v}{\partial t}\eta^2
v^{\beta-1}d\mu dt &=\frac{1}{\beta} \int_0^s\int_{M_t}
\frac{\partial( v^{\beta})}{\partial t}\eta^2 d\mu dt\\
&=\frac{1}{\beta}\int_{M_t}v^{\beta}\eta^2d\mu\Big|_0^s
-\frac{1}{\beta}\int_0^s\int_{M_t}
v^{\beta}\partial_t(\eta^2 d\mu)dt\\
&=\frac{1}{\beta}\int_{M_s}v^{\beta}\eta^2d\mu
-\frac{1}{\beta}\int_0^s\int_{M_t}v^{\beta}\left(2\eta
\frac{\partial \eta}{\partial t}-H^2\right)d\mu dt,
\end{aligned}
\end{equation}
where the last step we used the evolution of the volume form
\begin{equation}\label{volde}
\frac{\partial}{\partial t}d\mu=-H^2 d\mu.
\end{equation}
Substituting \eqref{int2} and \eqref{int3} into \eqref{int1} yields
\begin{equation}
\begin{aligned}\label{ibpart1}
&\int_0^s\int_{M_{t}}\left[2\langle\nabla v, \nabla\eta\rangle\eta
v^{\beta-1}+(\beta-1)\eta^2 v^{\beta-2}|\nabla v|^2\right]d\mu dt+
\frac{1}{\beta}\int_{M_s}v^{\beta}\eta^2d\mu\\
&\leq\frac{1}{\beta}\int_0^s\int_{M_t} v^{\beta} 2\eta
\frac{\partial \eta}{\partial t}d\mu dt +
\int_0^s\int_{M_t}|f+C|\eta^2 v^{\beta}d\mu dt.
\end{aligned}
\end{equation}
In the following, we want to construct the quantity
$(\frac{\partial}{\partial t}-\Delta)\eta$ to appear on the right
hand side of \eqref{ibpart1}. To achieve it, we find that
integrating by parts yields
\begin{equation*}
\begin{aligned}
\frac{1}{\beta}&\int_0^s\int_{M_t}v^{\beta} 2\eta \frac{\partial
\eta}{\partial
t}d\mu dt\\
&=\frac{1}{\beta}\int_0^s\int_{M_t}\left[v^{\beta}2\eta
\left(\frac{\partial}{\partial
t}-\Delta\right)\eta+v^{\beta}2\eta\Delta
\eta]\right)d\mu dt\\
&=\frac{1}{\beta}\int_0^s\int_{M_t}\left[v^{\beta}2\eta
\left(\frac{\partial}{\partial
t}-\Delta\right)\eta-2\nabla(v^{\beta}
\eta)\nabla\eta\right]d\mu dt\\
&=\frac{1}{\beta}\int_0^s\int_{M_t}\left[v^{\beta}2\eta
\left(\frac{\partial}{\partial t}-\Delta\right)\eta-
2v^{\beta}|\nabla\eta|^2-
2\beta\langle\nabla v, \nabla\eta\rangle\eta v^{\beta-1}\right]d\mu dt\\
&\leq\frac{1}{\beta}\int_0^s\int_{M_t}v^{\beta}2\eta
\left(\frac{\partial}{\partial t}-\Delta\right)\eta d\mu dt
-\int_0^s\int_{M_t}2\eta\langle\nabla v,\nabla\eta\rangle
v^{\beta-1}d\mu dt.
\end{aligned}
\end{equation*}
Hence \eqref{ibpart1} becomes
\begin{equation}
\begin{aligned}\label{ibpart2}
&\int_0^s\int_{M_t}\left[4\langle\nabla v, \nabla\eta\rangle\eta
v^{\beta-1}+(\beta-1)\eta^2 v^{\beta-2}|\nabla v|^2\right]d\mu dt+
\frac{1}{\beta}\int_{M_s}v^{\beta}\eta^2d\mu\\
&\leq\frac{1}{\beta}\int_0^s\int_{M_t}v^{\beta}2\eta
\left|\left(\frac{\partial}{\partial t}-\Delta\right)\eta\right|d\mu
dt+\int_0^s\int_{M_t}|f+C|\eta^2 v^{\beta}d\mu dt.
\end{aligned}
\end{equation}
Note that the Cauchy-Schwartz inequality implies
\begin{equation*}
\begin{aligned}
\int_0^s\int_{M_t}4\langle\nabla v, \nabla\eta\rangle\eta
v^{\beta-1}d\mu dt&\geq-2\epsilon^2\int_0^s\int_{M_t}
\eta^2v^{\beta-2}|\nabla v|^2d\mu dt\\
&\quad-\frac{2}{\epsilon^2}\int_0^s\int_{M_t}
v^{\beta}|\nabla\eta|^2 d\mu dt.
\end{aligned}
\end{equation*}
Hence \eqref{ibpart2} becomes
\begin{equation*}
\begin{aligned}
&\int_0^s\int_{M_{t}} (\beta-1-2\epsilon^2)\eta^2
v^{\beta-2}\abs{\nabla v}^2 d\mu dt +
\frac{1}{\beta}\int_{M_s}v^{\beta}\eta^2d\mu\\
&\leq \frac{1}{\beta}\int_0^s\int_{M_t}v^{\beta} 2\eta
\left|\left(\frac{\partial}{\partial t}-\Delta\right)\eta\right|d\mu
dt+\int_0^s\int_{M_t}|f+C|\eta^2v^{\beta}d\mu dt\\
&\quad+\frac{2}{\epsilon^2}\int_0^s\int_{M_t}
v^{\beta}|\nabla\eta|^2d\mu dt.
\end{aligned}
\end{equation*}
Noticing the relation
\[
|\nabla (v^{\beta/2})|^2=\frac{\beta^2}{4}v^{\beta-2}\abs{\nabla
v}^2,
\]
if we choose
$\epsilon^2=\frac{\beta-1}{4}$, then
\begin{equation*}
\begin{aligned}
&2\left(1-\frac{1}{\beta}\right)\int_0^s\int_{M_t} \eta^2
\abs{\nabla
(v^{\beta/2})}^2 d\mu dt+\int_{M_s}v^{\beta}\eta^2d\mu\\
&\leq\int_0^s\int_{M_t}v^{\beta}2\eta
\left|\left(\frac{\partial}{\partial t}
-\Delta\right)\eta\right|d\mu dt+\beta\int_0^s\int_{M_t}|f+C|\eta^2
v^{\beta}d\mu dt\\
&\quad+\frac{8\beta}{\beta-1}\int_0^s\int_{M_t} v^{\beta}|\nabla
\eta|^2d\mu dt.
\end{aligned}
\end{equation*}
Combining the above estimate with
\[
 \abs{\nabla(\eta v^{\beta/2})}^2\leq 2\eta^2
 \abs{\nabla (v^{\beta/2})}^2 + 2v^{\beta}\abs{\nabla\eta}^2
\]
implies
\begin{equation*}
\begin{aligned}
&\left(1-\frac{1}{\beta}\right)\int_0^s\int_{M_t}|\nabla(\eta
v^{\beta/2})|^2d\mu dt+\int_{M_s}v^{\beta}\eta^2d\mu\\
&\leq\int_0^s\int_{M_t}v^{\beta}2\eta\left|\left(\frac{\partial}{\partial
t}-\Delta\right)\eta\right|d\mu
dt+\beta\int_0^s\int_{M_t}|f+C|\eta^2
v^{\beta}d\mu dt\\
&\,\,\,\,\,\,+8\left(\frac{\beta}{\beta-1}+\frac{\beta-1}{\beta}\right)
\int_0^s\int_{M_t}v^{\beta}|\nabla\eta|^2d\mu dt.
\end{aligned}
\end{equation*}
It follows that, for some $\Lambda(\beta)\geq 1$ (for example, we
can choose $\Lambda (\beta) = 100\beta$ if $\beta\geq 2$),
\begin{equation*}
\begin{aligned}
&\int_0^s\int_{M_t}|\nabla(\eta v^{\beta/2})|^2d\mu dt+
\int_{M_s}v^{\beta}\eta^2d\mu\\
&\leq\Lambda(\beta)\int_0^s\int_{M_t}v^{\beta}\left\{2\eta
\left|\left(\frac{\partial}{\partial
t}-\Delta\right)\eta\right|+|\nabla\eta|^2\right\}d\mu dt\\
&\quad+\Lambda(\beta)\int_0^s\int_{M_t}|f+C|\eta^2v^{\beta}d\mu dt\\
&\leq \Lambda(\beta)\int_0^s\int_{M_t} v^{\beta}\left\{ 2\eta
\left|\left(\frac{\partial}{\partial t}-\Delta\right)
\eta\right|+|\nabla\eta|^2\right\}d\mu dt\\
&\quad+\Lambda(\beta)\left(||f||_{L^q(M\times[0,T))}||\eta^2
v^{\beta}||_{L^{\frac{q}{q-1}}(M\times[0,T))}
+\int_0^s\int_{M_t}C\eta^2v^{\beta}d\mu dt\right)\\
&=\Lambda(\beta)\int_0^s\int_{M_t}v^{\beta}\left\{2\eta
\left|\left(\frac{\partial}{\partial t}
-\Delta\right)\eta\right|+|\nabla\eta|^2\right\}d\mu dt\\
&\quad+\Lambda(\beta)\left(C_0||\eta^2
v^{\beta}||_{L^{\frac{q}{q-1}}(M\times[0,T))}+C||\eta^2
v^{\beta}||_{L^1(M\times[0,T))}\right)\\
&=: B.
\end{aligned}
\end{equation*}
From the above inequalities, we conclude
\begin{equation}
 \max_{0\leq s\leq T}\int_{M_s}\eta^2v^{\beta}d\mu\leq B
\end{equation}
and
\begin{equation}
 \int_0^T\int_{M_t}|\nabla(\eta v^{\beta/2})|^2 d\mu dt\leq B.
\end{equation}
If the function $(\eta v^{\frac{\beta}{2}})^{\frac{2(n-1)}{n-2}}$
satisfies conditions~\eqref{condition1} and~\eqref{condition2} for
any $t\in[0,T)$, applying Proposition \ref{propo1zh} to $\eta
v^{\beta/2}$, we have the following estimates
\begin{equation*}
\begin{aligned}
||&\eta^2v^{\beta}||^{(n+2)/n}_{L^{(n+2)/n}(M\times[0,T))}\\
&=||\eta v^{\beta/2}||^{2(n+2)/n}_{L^{2(n+2)/n}(M\times[0,T))}\\
&\leq c_n\max_{0\leq t\leq T} ||\eta
v^{\beta/2}||^{4/n}_{L^2(M_{t})}\\
&\quad\times\left(||\nabla(\eta
v^{\beta/2})||^2_{L^2(M\times[0,T))}+\max_{0\leq t\leq T}
||\eta v^{\beta/2}||^2_{L^2(M_t)}||H||^{n+2}_{L^{n+2}(M\times[0,T))}\right)\\
&\leq c_nB^{2/n}\left(B+B||H||^{n
+2}_{L^{n+2}(M\times[0,T))}\right)\\
&=c_nB^{\frac{n+2}{n}}\left(1+||H||^{n+2}_{L^{n+2}(M\times[0,T))}\right).
\end{aligned}
\end{equation*}
For the convenient of writing,  we let $S:=M\times [0, T)$ and let
the norm $||\cdot||_{L^p(M\times[0,T))}$ be shortly denoted by
$||\cdot||_{L^p(S)}$. Then by the definition of $B$, the previous
estimate can be rewritten as
\begin{equation}
\begin{aligned}\label{superc}
||\eta^2v^{\beta}||_{L^{\frac{n+2}{n}}(S)}
&\leq c_nB\left(1+||H||^{n+2}_{L^{n+2}(S)}\right)^{\frac{n}{n+2}}\\
&=c_nC_1\Lambda(\beta)\int_0^T\int_{M_t}v^{\beta}\left\{2\eta
\left|\left(\frac{\partial}{\partial t}-\Delta\right)\eta\right|
+|\nabla\eta|^2\right\}d\mu dt\\
&\quad+c_nC_1\Lambda(\beta)\left(C_0||\eta^2
v^{\beta}||_{L^{\frac{q}{q-1}}(S)}+C||\eta^2 v^{\beta}||_{L^1(S)}
\right).
\end{aligned}
\end{equation}
Since $1<\frac{q}{q-1}<\frac{n+2}{n}$, we have the interpolation
inequality
\[
||\eta^2 v^{\beta}||_{L^{\frac{q}{q-1}}(S)}\leq\epsilon||\eta^2
v^{\beta}||_{L^{\frac{n+2}{n}}(S)}+\epsilon^{-\nu}||\eta^2
v^{\beta}||_{L^1(S)},
\]
where $\nu=\frac{n+2}{2q-(n+2)}$. Then plugging this inequality into
\eqref{superc} yields
\begin{equation*}
\begin{aligned}
&[1-c_n\Lambda(\beta)C_0C_1\epsilon]||
\eta^2v^{\beta}||_{L^{\frac{n+2}{n}}(S)}\\
&\leq c_nC_1\Lambda(\beta)\left[(C_0\epsilon^{-\nu}+C)||\eta^2
v^{\beta}||_{L^1(S)}+\left|\left|v^{\beta}\left(|\nabla\eta|^2+2\eta
\left(\frac{\partial}{\partial
t}-\Delta\right)\eta\right)\right|\right|_{L^1(S)}\right].
\end{aligned}
\end{equation*}
If we further choose $\epsilon=\frac{1}{2\Lambda (\beta)c_nC_0C_1}$,
then
\begin{equation*}
\begin{aligned}
&||\eta^2v^{\beta}||_{L^{\frac{n+2}{n}}(S)}\\
&\leq 2c_nC_1\Lambda (\beta)\left[C_0\left(2\Lambda
(\beta)c_nC_0C_1\right)^{\nu}+C\right]||\eta^2 v^{\beta}||_{L^1(S)}\\
&\quad+2c_nC_1\Lambda
(\beta)\left|\left|v^{\beta}\left(|\nabla\eta|^2+
2\eta\left(\frac{\partial}{\partial t}-\Delta\right)\eta\right)\right|\right|_{L^1(S)}\\
&\leq
C_a(n,q,C_0,C_1,C)\Lambda(\beta)^{1+\nu}\left|\left|v^{\beta}\left(\eta^2+
|\nabla\eta|^2+2\eta\left(\frac{\partial}{\partial
t}-\Delta\right)\eta\right)\right|\right|_{L^1(S)},
\end{aligned}
\end{equation*}
where
\[
C_a(n,q,C_0,C_1,C)=(2c_nC_0C_1)^{1+\nu}+2c_nC_1(C+1).
\]
Therefore we obtain a soft reverse H\"{o}lder inequality
\begin{equation*}
\begin{aligned}
&||\eta^2v^{\beta}||_{L^{\frac{n+2}{n}}(S)}\\
&\le
C_a(n,q,C_0,C_1,C)\Lambda(\beta)^{1+\nu}\left|\left|v^{\beta}\left(\eta^2+
|\nabla\eta|^2+2\eta\left|\left(\frac{\partial}{\partial
t}-\Delta\right)\eta\right|\right)\right|\right|_{L^1(S)}
\end{aligned}
\end{equation*}
provided the function $(\eta
v^{\frac{\beta}{2}})^{\frac{2(n-1)}{n-2}}$ satisfies
conditions~\eqref{condition1} and~\eqref{condition2} for any
$t\in[0,T)$.
\end{proof}

Next, we shall show that an $L^{\infty}$-norm of $v$ over a smaller
set can be bounded by an $L^{\beta}$-norm of $v$ on a bigger set,
where $\beta\ge2$. Fix $x_0\in N^{n+1}$. Consider the following
space-time sets:
\[
D=\cup_{0\le t\le 1}(B(x_0,1)\cap M_t);\quad D^{'}=
\cup_{\frac{1}{12}\le t\leq 1}(B(x_0,\frac{1}{2^p})\cap M_t),
\]
where $p\geq1$ is large enough, depending on $n$, $\max_{x\in
M_0}|A|$ and $N$, to be described later. Then, we have the following
Harnack inequality.
\begin{lemma}\label{moser}
Consider the equation \eqref{keyeq} with $T=1$. Let us denote by $
\lambda=\frac{n+2}{n}$, let $q>\frac{n +2}{2}$ and $\beta\geq 2$.
Then, there exists a constant $C_b=C_b(n,q,\beta,C_0,C_1,C)$ such
that
\begin{equation}\label{firstM}
||v||_{L^{\infty}(D^{'})}\leq C_b||v||_{L^{\beta}(D)}.
\end{equation}
In the above inequalities, $C_0$ and $C_1$ are defined by
(\ref{Czero}).
\end{lemma}
\begin{remark}
In fact, in Lemma \ref{moser} we can choose
\begin{equation}\label{eq-Cb}
C_{b}(n,p,q,\beta,C_0,C_1,C)=(4^p\lambda^{1+\nu}
C_z\beta^{1+\nu})^{\frac{n^2}{\beta}},
\end{equation}
where
\begin{equation}\label{eq-Cb2}
C_z(n,q,C_0,C_1,C):=4^2\times 100^{1+\nu}c_nC_a(n,q,C_0,C_1,C).
\end{equation}
\end{remark}

\begin{proof}[Proof of Lemma~\ref{moser}]
The proof of this lemma, using Lemma \ref{sorevho} and Moser
iteration, is similar to that of Lemma 5.2 in \cite{Le-Sesum}. If we
set
\[
D_k=\cup_{t_k\le t\le 1}(B(x_0,r_k)\cap M_t),
\]
where
\[
r_k=\frac{1}{2^p}+\frac{1}{2^{p+k+1}}; \quad t_k=
\frac{1}{12}\left(1-\frac{1}{4^{p+k}}\right),
\]
then
\[
\rho_k:=r_{k-1}-r_k=\frac{1}{2^{p+k+1}}; \quad t_k-t_{k-1}=\rho^2_k.
\]
Following Ecker \cite{Ecker95}, we choose a test function $\eta_k
=\eta_k(t,x)$ of the form
\begin{equation}\label{testk}
\eta_k(t,x)=\varphi_{\rho_k}(t)\times \psi_{\rho_k}(|x-x_0|^2),
\end{equation}
where the function $\varphi_{\rho_k}$ satisfies
\[
 \varphi_{\rho_k}(t)=
\left\{
 \begin{alignedat}{1}
1 \quad &\text{if}~ t_{k}\leq t\leq 1, \\\
\in [0, 1] \quad &\text{if}~ t_{k-1}\leq t\leq t_k, \\\
0 \quad &\text{if}~ t\leq t_{k-1}.
 \end{alignedat}
 \right.
 \]
and
\[
 \abs{\varphi^{'}_{\rho_k}}(t)\leq\frac{c_n}{\rho^2_k};
\]
whereas the function $\psi_{\rho_k}(s)$ satisfies
\[
 \psi_{\rho_k}(s)=
\left\{
 \begin{alignedat}{1}
0 \quad &\text{if}~  s\geq r_{k-1}^2, \\\
\in [0, 1]  \quad &\text{if}~ r^2_k\leq s\leq r^2_{k-1}, \\\
1 \quad &\text{if}~ s\leq r^2_k.
 \end{alignedat}
 \right.
\]
and
\[
 \abs{\psi^{'}_{\rho_k}}(s)\leq \frac{c_n}{\rho^2_k}.
\]
By the definition, we can easily check that
\[
0\leq \eta_k\leq 1;\quad\eta_k\equiv
1~\text{in}~D_k;\quad\eta_k\equiv 0~\text{outside}~D_{k-1}
\]
and
\begin{equation}\label{etaDk}
\sup_{t\in [0,1]}\sup_{x\in M_t}\left(\eta^2_k+
|\nabla\eta_k|^2+2\eta_k\left|\left(\frac{\partial}{\partial
t}-\Delta\right)\eta_k\right|\right)\leq\frac{c_n}{\rho^2_k}=c_n4^{p+k+1},
\end{equation}
where $\eta_k=\eta_k(t,x)$. Here we claim that if $\beta\geq 2$, let
$\Lambda(\beta) =100\beta$. Now using Lemma \ref{sorevho}, we have
\begin{equation}
\begin{aligned}\label{etaDkc}
||v^{\beta}||_{L^{\frac{n+2}{n}}(D_k)}
&\leq||\eta_k^2v^{\beta}||_{{L^{\frac{n+2}{n}}}(M\times[0,T))}\\
&\leq C_a\Lambda(\beta)^{1+
\nu}\int_0^T\int_{M_t}v^{\beta}\left(\eta_k^2+|\nabla\eta_k|^2+
2\eta_k\left|\left(\frac{\partial}{\partial t}-\Delta\right)\eta_k\right|\right)d\mu dt\\
&\leq c_n4^{p+k+1}C_a\Lambda(\beta)^{1+\nu}\int_{D_{k-1}}v^{\beta}d\mu dt\\
&:=C_z(n,q,C_0,C_1,C) 4^{p+k-1}\beta^{1+
\nu}||v^{\beta}||_{L^1(D_{k-1})},
\end{aligned}
\end{equation}
as long as $p$ is sufficient large so that the function $(\eta_k
v^{\frac{\beta}{2}})^{\frac{2(n-1)}{n-2}}$ naturally satisfies
conditions~\eqref{condition1} and~\eqref{condition2} for any
$t\in[0,T)$, where the third line results from the inequality
\eqref{etaDk} and the definition of $\rho_k$. Note that we can
choose
\[
\Lambda(\beta)=100\beta,\quad C_z(n,q,C_0,C_1,C)=4^2\times
100^{1+\nu}c_nC_a.
\]

Now we explain why $(\eta_k
v^{\frac{\beta}{2}})^{\frac{2(n-1)}{n-2}}$ naturally satisfies
conditions~\eqref{condition1} and~\eqref{condition2} for any
$t\in[0,T)$ when $p$ is large enough. In fact under mean curvature
flow, we observe that
\[
Vol_{g(t)}(B(R))\leq Vol_{g(0)}(B(R))
\]
for any $t\in[0,T)$ by \eqref{volde}.  For $g(0)$, there is a
non-positive constant $K$ depending on $n$, $\max_{x\in M_0}|A|$ and
$N$ such that the sectional curvature of $M_0$ is bounded from below
by $K$. Then the Bishop-Gromov volume comparison theorem implies
\[
Vol_{g(0)}(B(R))\leq Vol_K(B(R)),
\]
where $Vol_K(B(R))$ denotes the volume of the ball with radius $R$
in the $n$-dimensional complete simply connected space form with
constant curvature $K$. Hence
\[
Vol_{g(t)}(B(R))\leq Vol_K(B(R)).
\]
Therefore, we can choose $R$ sufficient small such that
\[
b^2(1-\alpha)^{-2/n}\left(\omega_n^{-1}Vol_K(B(R))\right)^{2/n}\leq
1
\]
and
\[
2\rho_0\leq\bar{R}(M),
\]
where $\rho_0$ is defined by \eqref{canshu1}. Here the sufficient
small $R$ can be achieved by choosing a sufficient large $p$. Hence
this explanation shows that the function $(\eta_k
v^{\frac{\beta}{2}})^{\frac{2(n-1)}{n-2}}$ naturally satisfies
conditions~\eqref{condition1} and~\eqref{condition2} for any
$t\in[0,T)$ as long as $p$ is large enough.

For simplicity, let us denote  by $C_z=C_z(n,q,C_0,C_1,C)$. Then
\eqref{etaDkc} can be simplified as follows:
\begin{equation}\label{firstRH}
||v||_{L^{\frac{n+2}{n}\beta}(D_k)}\leq
C^{\frac{1}{\beta}}_z4^{\frac{p+k-1}{\beta}}\beta^{\frac{1+
\nu}{\beta}}||v||_{L^{\beta}(D_{k-1})}.
\end{equation}
This inequality is the key estimate for our Moser iteration process.
If we set $\lambda =\frac{n+2}{n}$ and replace $\beta$ by
$\lambda^{k-1}\beta$ in \eqref{firstRH}, then
\begin{equation*}
\begin{aligned}
||v||_{L^{\beta \lambda^k}(D_{k})}&\leq
C^{\frac{1}{\beta\lambda^{k-1}}}_z4^{\frac{p+k-1}{\beta\lambda^{k-1}}}
(\lambda^{k-1}\beta)^{\frac{1+\nu}{\beta\lambda^{k-1}}}
||v||_{L^{\beta\lambda^{k-1}}(D_{k-1})}\\
&\leq
C^{\frac{1}{\beta\lambda^{k-1}}}_z(4^p)^{\frac{k-1}{\beta\lambda^{k-1}}}
(\lambda^{k-1}\beta)^{\frac{1+\nu}{\beta\lambda^{k-1}}}
||v||_{L^{\beta\lambda^{k-1}}(D_{k-1})}.
\end{aligned}
\end{equation*}
It follows by iteration that for all $k_0\geq 0$
\begin{equation}
\begin{aligned}\label{secondRH}
||v||_{L^{\beta\lambda^k}(D_k)} &\leq
\left(C^{\frac{1}{\beta}}_z\beta^{\frac{1+\nu}{\beta}}\right)^{\sum^{k-1}_{j=k_0}
\frac{1}{\lambda^j}}\left(4^{\frac{p}{\beta}}\cdot
\lambda^{\frac{1+\nu}{\beta}}\right)^{\sum^{k-1}_{j=k_0}
\frac{j}{\lambda^j}}||v||_{L^{\beta\lambda^{k_0}}(D_{k_0})} \\
&\leq C_b(n,p,q,\beta, C_0,
C_1,C)||v||_{L^{\beta\lambda^{k_0}}(D_{k_0})},
\end{aligned}
\end{equation}
where we can choose
\[
C_b(n,p,q,\beta,C_0,C_1,C)=(4^p\lambda^{1+\nu}
C_z\beta^{1+\nu})^{\frac{n^2}{\beta}},
\]
since
\[
\sum^{\infty}_{j=0}\frac{j}{\lambda^j}=\frac{\lambda}{(\lambda
-1)^2}=\frac{n(n+2)}{4}\leq n^2.
\]
Note that $D_0\subseteq D$ and $D^{'}\subset D_{k}$ for all positive
integer $k$. So inequality \eqref{secondRH} yields
\begin{equation*}
\begin{aligned}
||v||_{L^{\beta\lambda^k}(D^{'})}&\leq C_b(n,p,q,\beta,C_0,C_1,C)
||v||_{L^{\beta}(D_0)}\\
&\leq C_b(n,p,q,\beta,C_0,C_1,C) ||v||_{L^{\beta}(D)}.
\end{aligned}
\end{equation*}
Letting $k\to \infty$, we get
\[
||v||_{L^{\infty}(D^{'})}\leq C_b(n,p,q,\beta,C_0,C_1,C)
||v||_{L^{\beta}(D)}.
\]
\end{proof}

\section{Bounding the second fundamental form}\label{sec5}
In this section, we will prove that the second fundamental form
$A(\cdot,t)$ can be bounded by its certain integral sense. First we
establish a rescaled version.
\begin{proposition}\label{smallcont}
Let $F_t:M^n\rightarrow N^{n+1}$ be a solution to the mean curvature
flow \eqref{MCF1} of compact hypersurfaces of a locally symmetric
Riemannian manifold $N^{n+1}$ with bounded geometry. If there exists
a universal constant $c_0$ such that
\begin{equation}
 \int_0^1\int_{M_t}|A|^{n+3}d\mu dt\leq c_0,
\end{equation}
 then
\begin{equation}
 \sup_{\frac12\leq t\leq 1}\sup_{x\in M_t}|A(x,t)|\leq 1.
\end{equation}
\end{proposition}

\begin{proof}[Proof of Proposition~\ref{smallcont}]
We follow the lines of the proof of Lemma 4.1 in~\cite{Le}. In
Section~\ref{sec2}, we have the following evolution equation:
\begin{equation*}
\begin{aligned}
\left(\frac{\partial}{\partial t}-\Delta\right)|A|^2&=-2|\nabla A|^2
+2|A|^2\left(|A|^2+\bar{R}ic(\nu,\nu)\right)\\
&\,\,\,\,\,\,-4\left(h^{ij}h^m_j\bar{R}_{mli}\
^l-h^{ij}h^{lm}\bar{R}_{milj}\right)\\
&\,\,\,\,\,\,-2h^{ij}(\bar{\nabla}_j\bar{R}_{0li}\
^l+\bar{\nabla}_l\bar{R}_{0ij}\ ^l).
\end{aligned}
\end{equation*}
Since $N$ is locally symmetric by assumption, we have
$\bar{\nabla}\bar{R}m= 0$. So the above equation becomes
\begin{equation}
\begin{aligned}\label{eq-ev-rim}
\left(\frac{\partial}{\partial t}-\Delta\right)|A|^2&=-2|\nabla A|^2
+2|A|^2\left(|A|^2+\bar{R}ic(\nu,\nu)\right)\\
&\,\,\,\,\,\,-4\left(h^{ij}h^m_j\bar{R}_{mli}\
^l-h^{ij}h^{lm}\bar{R}_{milj}\right)
\end{aligned}
\end{equation}
Furthermore, using the bounds on the geometry of $N$  we have
\[
\left(\frac{\partial}{\partial
t}-\Delta\right)|A|^2\le2|A|^4+C|A|^2,
\]
where $C$ is a positive constant, depending on $N$. If we set
$v=|A|^2$ and $f=2|A|^2$, then
\[
\left(\frac{\partial}{\partial t}-\Delta\right)v\le (f+C)v.
\]

Now we will use Lemma~\ref{moser} to the above equation. Here we let
$q=\frac{n+3}{2}$ and $\beta=\frac{n+3}{2}$. \eqref{nueq} and
\eqref{Ca} imply
\[
C_a=(2c_nC_0C_1)^{n+3}+2c_nC_1(C+1).
\]
Meanwhile from \eqref{eq-Cb} and \eqref{eq-Cb2}, we have
\[
C_b=c_n(p)C_z^{\frac{2n^2}{n+3}}=c_n(p)C_a^{\frac{2n^2}{n+3}}
=c_n(p)[(C_0C_1)^{n+3}+C_1(C+1)]^{\frac{2n^2}{n+3}},
\]
where $c_n(p)$ only depends on $n$ and $p$. We also know that
\begin{equation}\label{budengshi1}
\begin{aligned}
C_0&=2\left[\int_0^1\int_{M_t}|A|^{2q}d\mu
dt\right]^{1/q}\\
&\leq 2c_0^{\frac{2}{n+3}}
\end{aligned}
\end{equation}
and
\begin{equation}
\begin{aligned}\label{budengshi2}
C_1&=\left[1+\left(\int_0^1\int_{M_t}|H|^{n+3}d\mu dt\right)
^{\frac 23}\right]^{\frac{n}{n+2}}\\
&\leq \left[1+\left(\int_0^1\int_{M_t}n^{\frac
{n+3}{2}}|A|^{n+3}d\mu dt\right)
^{\frac 23}\right]^{\frac{n}{n+2}}\\
&\leq \left(1+n^{\frac{n+3}{3}}c_0^{\frac
23}\right)^{\frac{n}{n+2}},
\end{aligned}
\end{equation}
where we used inequality $|H|^2\leq n|A|^2$. Hence by Lemma
\ref{moser}, and using \eqref{budengshi1} and \eqref{budengshi2}, we
have
\begin{equation}
\begin{aligned}
||v||_{L^{\infty}(D^{'})}&\leq
C_b||v||_{L^{\beta}(D)}\\
&=c_n(p)[(C_0C_1)^{n+3}+C_1(C+1)]^{\frac{2n^2}{n+3}}||v||_{L^{\beta}(D)}\\
&\leq
c_n(p)[(C_0C_1)^{n+3}+C_1(C+1)]^{\frac{2n^2}{n+3}}c_0^{\frac{2}{n+3}}\\
&\leq 1
\end{aligned}
\end{equation}
if $c_0$ is sufficient small. Note that $c_n(p)$ does not depend on
$c_0$.
\end{proof}

Using Proposition \ref{smallcont}, and following the proof of
Proposition 1.1 in \cite{Le}, we have the following key result for
proving Theorem \ref{wuN2}.
\begin{proposition}\label{smallcont2}
Let $F_t:M^n\rightarrow N^{n+1}$ be a solution to the mean curvature
flow \eqref{MCF1} of compact hypersurfaces of a locally symmetric
Riemannian manifold $N^{n+1}$ with bounded geometry. For all
$\lambda\in(0,1]$, there exists a constant $c_{\lambda}$ such that
for all $T\geq \lambda$, we have
\[
\sup_{x\in M_T}|A(x,T)|\leq c_{\lambda}
\left(1+\int^T_0\int_{M_t}|A|^{n+3}d\mu dt\right).
\]
\end{proposition}
\begin{proof}[Proof of Theorem~\ref{smallcont2}]
The proof is the same as that of Proposition 1.1 in \cite{Le}.
\end{proof}

\section{Proof of Theorem \ref{wuN2}}\label{sec7}
In this section, we will follow the Le's method in \cite{Le} and
prove Theorem \ref{wuN2} in introduction. The proof is the same as
the Le's proof. For the completeness, we still give the details of
the proof here.
\begin{proof}[Proof of Theorem \ref{wuN2}]
Fix $\tau_1<T$ such that $0<\tau_1<1$. By Proposition
\ref{smallcont2}, for any $t\geq\tau_1$, there is a universal
constant $c$ depending on $\tau_1$, such that
\begin{equation}\label{tauineq}
\sup_{x\in M_t}|A(x,t)|\leq
c(\tau_1)\left(1+\int_0^t\int_{M_s}|A|^{n+3}d\mu ds\right).
\end{equation}
We set $f(t)=\sup_{x\in M_t}|A(x,t)|$, $\Psi(s)=s\log(a+s)$ and
\[
G(s)=\int_{M_s} \frac{|A|^{n+2}}{\log(a+|A|)}d\mu,
\]
where $a$ is a fixed constant and $a\geq 1$.

Note that $\Psi$ is an increasing function. Then \eqref{tauineq} can
be read as
\begin{eqnarray*}
 f(t)&\leq&c(\tau_1)\left(1+\int_0^t \int_{M_s}\Psi(|A|)
\frac{|A|^{n+2}}{\log(a+|A|)}d\mu ds\right) \\
&\leq& c(\tau_1)\left[1+\int_0^t\Psi\left(\sup_{x\in
M_s}|A(x,s)|\right)\int_{M_s}
\frac{|A|^{n+2}}{\log(a+|A|)}d\mu ds\right]\\
&=&c(\tau_1)\left(1+\int_0^t\Psi (f(s))G(s) ds\right).
\end{eqnarray*}
If we set
\[
 h(t)=c(\tau_1)\left(1+\int_0^t\Psi(f(s))G(s)ds\right),
\]
then for $t\geq \tau_1$
\[
f(t)\leq h(t)
\]
and
\[
 h^{'}(t)=c(\tau_1)\Psi(f(t))G(t)\leq c(\tau_1)\Psi(h(t)) G(t).
\]
If we further set
\[
\tilde{\Psi}(y)=\int_c^y \frac{1}{\Psi (s)} ds,
\]
then for $t\geq \tau_1$,
\[
\tilde{\Psi}(h(t))-\tilde{\Psi}(h(\tau_1))\leq
c(\tau_1)\int_{\tau_1}^tG(s)ds\leq c(\tau_1)\int_0^T G(s)ds<\infty.
\]
Note that $h(\tau_1)$ is finite. Therefore
\begin{equation}\label{zuihou}
\sup_{\tau_1\leq t<T}\tilde{\Psi}(h(t))\leq\tilde{\Psi}
(h(\tau_1))+c(\tau_1)\int_0^TG(s)ds<\infty.
\end{equation}
Since
$\int_c^{\infty}\frac{ds}{\Psi(s)}=\int_c^{\infty}\frac{ds}{s\log(a+s)}=\infty$,
by \eqref{zuihou} we deduce that
\[
\sup_{\tau_1\leq t<T} h(t)<\infty.
\]
Hence
\[
\sup_{\tau_1\leq t<T} f(t)<\infty.
\]
Namely, $\sup_{x\in M_t}|A(x,t)|<\infty$ for $0\leq t<T$. By Theorem
\ref{Coopound} in introduction, the mean curvature flow can be
extended past $T$.
\end{proof}

Just as Le's said in \cite{Le}, from the proof of Theorem
\ref{wuN2}, we have another subcritical quantities concerning the
second fundamental and obtain a similar extension result. For
example: We choose $\Psi(s)=s\log(a+s^b)$  and
\[
G(s)=\int_{M_s} \frac{|A|^{n+2}}{\log(a+|A|^b)}d\mu,
\]
where $b$ is a fixed constant and $b\geq 1$. Notice that
$\int_c^{\infty}\frac{ds}{s\log(a+s^b)}=\infty$. Following the proof
of Theorem \ref{wuN2}, we have
\begin{theorem}\label{wuspec}
Suppose $T<\infty$ is the first singularity time for a mean
curvature flow \eqref{MCF1} of compact hypersurfaces of a locally
symmetric Riemannian manifold $N^{n+1}$ with bounded geometry. If
\[
\int_0^T\int_{M_t}\frac{|A|^{n+2}}{\log(a+|A|^b)}d\mu dt<\infty,
\]
where $a$ and $b$ are fixed constants and $a, b\geq 1$, then this
flow can be extended past time $T$.
\end{theorem}

\section{Proof of Corollary \ref{wuN1}}\label{sec8}
In this section, we shall apply Theorem \ref{wuN2} to give the proof
of Corollary \ref{wuN1} stated in the introduction.

\begin{proof}[Proof of Corollary \ref{wuN1}]

By H\"{o}lder's inequality, we know that $||A||_{L^\alpha(M\times
[0,T))}<\infty$, where $\alpha>n+2$, implies
$||A||_{L^{n+2}(M\times[0,T))} <\infty$. So we only need to prove
Corollary \ref{wuN1} for $\alpha=n+2$.

From Theorem \ref{wuN2}, we see that if $T<\infty$ is the first
singularity time for a mean curvature flow (\ref{MCF1}) of a compact
hypersurface of a locally symmetric Riemannian manifold $N^{n+1}$
with bounded geometry, and if
\[
\int_0^T\int_{M_t}\frac{|A|^{n+2}}{\log(a+|A|)}d\mu dt<\infty
\]
where $a$ is a fixed constant and $a\geq 1$, then the mean curvature
flow can be extended past $T$. Now if we choose $a=100$, then
\[
\frac{1}{\log(100+|A|)}\leq \frac 12.
\]
Namely,
\[
\int_0^T\int_{M_t}\frac{|A|^{n+2}}{\log(100+|A|)}d\mu dt\leq\frac
12\int_0^T\int_{M_t}|A|^{n+2}d\mu dt.
\]
Hence the assumption of Corollary \ref{wuN1} guarantees that the
mean curvature flow can be extended past $T$.
\end{proof}

\section*{Acknowledgment}
The author would like to thank Dr. E.-T. Zhao for his interest and
pointing out a mistake in proving Corollary \ref{wuN1} on an earlier
version of this manuscript.

\end{document}